\documentclass[12pt]{amsart}

\usepackage{amssymb, amsfonts, amsthm}
\usepackage{graphicx}
\numberwithin{equation}{section}

\newtheorem{theorem}{Theorem}[section]
\newtheorem{lemma}[theorem]{Lemma}
\newtheorem{corollary}[theorem]{Corollary}

\theoremstyle{definition}
\newtheorem{definition}[equation]{Definition}

\newcommand{\A}{{\mathcal A}}

\newcommand{\K}{{\mathcal K}}
\newcommand{\es}{{\mathcal S}}

\newcommand{\IC}{{\mathbb C}}

\newcommand{\D}{{\mathbb D}}

\newcommand{\ra}{{\rightarrow}}

\newcounter{minutes}
\divide\time by 60
\newcounter{hours}
\multiply\time by 60 \addtocounter{minutes}{-\time}

\begin{document}

\title[A generalization of close-to-convex functions]
{On a generalization of close-to-convex functions}

\author{S.K. Sahoo${}^{\mathbf{*}}$ and N.L. Sharma}
\address{S.K. Sahoo and N.L. Sharma, 
Discipline of Mathematics,
Indian Institute of Technology Indore,
Indore 452 017, India}
\email{swadesh@iiti.ac.in}
\email{sharma.navneet23@gmail.com}


\thanks{* The corresponding author}


\begin{abstract}
A motivation comes from {\em M. Ismail and et al.:
A generalization of starlike functions, 
Complex Variables Theory Appl., 14 (1990), 77--84} to study
a generalization of close-to-convex functions 
by means of a $q$-analog of a difference operator acting on analytic functions
in the unit disk $\mathbb{D}=\{z\in \mathbb{C}:\,|z|<1\}$. 
We use the terminology {\em $q$-close-to-convex functions} for the $q$-analog of
close-to-convex functions. The $q$-theory has wide applications in
special functions and quantum physics which makes the study interesting and pertinent in this field. 
In this paper, 
we obtain some interesting results concerning conditions 
on the coefficients of power series of functions analytic in the unit disk which ensure that 
they generate functions in the $q$-close-to-convex family. As a result we find 
certain dilogarithm functions that are contained in this family.
Secondly, we also study the famous Bieberbach conjecture problem on coefficients of analytic
$q$-close-to-convex functions. This produces several power series of analytic functions convergent to
basic hypergeometric functions.   
 
\end{abstract}


\maketitle

\begin{center}
\texttt{File:~\jobname .tex, printed: \number\day-\number\month-\number\year,
\thehours.\ifnum\theminutes<10{0}\fi\theminutes}
\end{center}

\section{Introduction}

Denote by $\A$, the class of functions $f(z)$, normalized by 
$f(0)=0=f'(0)-1$, that are analytic in the unit disk 
$\D:=\{z\in\IC:\,|z|<1\}$. 
In other words, the functions $f(z)$ in $\A$ have the power series representation
$$f(z)=z+\sum_{n=2}^\infty a_n z^n, \quad z\in\D.
$$
We denote by $\es$, the class of 
univalent (i.e. analytic and one-one) functions in $\D$.
Denote by $\es^*$, the subclass consisting of functions $f(z)$ in $\es$
that are starlike with respect to the origin, i.e.
$tw\in f(\D)$ whenever $t\in[0,1]$ and $w\in f(\D)$.
Analytically, it is well-known that $f(z)\in \es^*$ if and only if
$${\rm Re}\,\left(\frac{zf'(z)}{f(z)}\right)>0, \quad z\in \D\,;
$$
and a function $f(z)\in\A$ is said to be {\em close-to-convex} 
if there exists $g(z)\in \es^*$ such that
$${\rm Re}\,\left(\frac{zf'(z)}{g(z)}\right)>0, \quad z\in \D\,.
$$
Then we say that $f\in\K$ with the function $g$.
The class of close-to-convex functions defined in the unit disk is denoted by $\K$.
One can easily verify the fact that $\es^*\subset \K\subset \es$ 
(see for instance \cite{Dur83}). For several interesting geometric properties
of these classes, one can refer to the standard books \cite{Good83,Pom75}.

A $q$-analog of the class of starlike functions was first introduced 
in \cite{IMS90} by means 
of the {\em $q$-difference operator} $(D_qf)(z)$ acting on functions 
$f(z)\in \A$ defined by
\begin{equation} \label{q-operator}
(D_qf)(z)=\frac{f(z)-f(qz)}{z(1-q)}, \quad z\in \D\setminus\{0\}, \quad (D_qf)(0)=f'(0),
\end{equation}
where $q\in (0,1)$. Note that the $q$-difference operator 
plays an important role in the theory of hypergeometric series and quantum physics (see for instance 
\cite{And74,Ern02,Fin88,Kir95,Sla66}). One can clearly see that $(D_qf)(z)\to f'(z)$ as $q\to 1^{-}$. 
This difference operator helps us to generalize the class of starlike functions $\es^*$ analytically.
We denote by $\es_q^*$, the class of functions in this generalized family. For the sake of convenience, 
we also call functions in $\es_q^*$ the {\em $q$-starlike functions}. This is defined as follows: 
\begin{definition}\label{def1}
A function $f\in\A$ is said to belong to the class $\es_q^*$ if
$$\left|\frac{z}{f(z)}(D_qf)(z)-\frac{1}{1-q}\right|\leq \frac{1}{1-q}, \quad z\in \D\,.
$$
\end{definition}
Clearly, when $q\ra 1^{-}$, the class $\es_q^*$ will coincide with $\es^*$.  

As $\es_q^*$ generalizes $\es^*$ in the above manner, a similar form of $q$-analog 
of close-to-convex functions was expected and it is defined in the following form 
(see \cite{RS12}).  
\begin{definition}\label{def2}
A function $f\in\A$ is said to belong to the class $\K_q$ if there exists $g\in \es^*$ such that
$$\left|\frac{z}{g(z)}(D_qf)(z)-\frac{1}{1-q}\right|\leq \frac{1}{1-q}, \quad z\in \D\,.
$$
Then we say that $f\in\K_q$ with the function $g$.
\end{definition}
In \cite{RS12}, the authors have investigated some basic properties of functions
that are in $\K_q$. Some of these results are also recalled in this paper
in order to exhibit their interesting consequences.
As $(D_qf)(z)\to f'(z)$, as $q\to 1^{-}$, we observe in the limiting sense that the closed disk 
$|w-(1-q)^{-1}|\le (1-q)^{-1}$ 
becomes the right half-plane ${\rm Re}\,(zf'(z)/g(z))>0$ and hence
the class $\K_q$ clearly reduces to $\K$. In this paper, we refer to the functions in the class
$\K_q$ the {\em $q$-close-to-convex functions}. 
For the sake of convenience, we use the notation $\es_q^*$ instead of the notation $PS_q$ 
used in \cite{IMS90} and $\K_q$ instead of $PK_q$ used in \cite{RS12}.
It is easy to see that
$\es_q^*\subset \K_q$ for all $q\in (0,1)$.
Clearly, one can easily see from the above discussion that 
$$\bigcap_{0<q<1}\K_q\subset  \K\subset \es.
$$

Our main aim in this paper is to consider the following two ideas.

The first idea has its genesis in the work of Frideman \cite{Fri46}. 
He proved that there are only nine functions in the class $\es$ whose
coefficients are rational integers. They are
$$z,\quad \frac{z}{1\pm z}, \quad \frac{z}{1\pm z^2}, \quad \frac{z}{(1\pm z)^2} 
\quad \frac{z}{1\pm z+z^2}.
$$
It is easy to see that these functions map the unit disk $\D$ onto starlike domains.
Using the idea of MacGregor \cite{Mac69}, we derive some sufficient conditions
for functions to be in $\K_q$ whose coefficients are connected with certain monotone properties.
These sufficient conditions help us to examine functions of dilogarithm types \cite{Kir95, Zag07}
which are in the $\K_q$ family.
Certain special functions, which are in the starlike and close-to-convex family, have been
well-investigated in \cite{HPV10,MS61,MM90,Pon97,PV01,RS86,Sil93}.

The second idea deals with the famous Bieberbach conjecture problem in analytic
univalent function theory \cite{DeB85,Dur83}.
A necessary and sufficient condition for a function $f(z)$ to be in $S_q^*$ 
is obtained in \cite{IMS90} by means of an integral representation of the function $zf'(z)/f(z)$ 
which yields the maximum moduli of coefficients of $f$. Using this condition, the Bieberbach
conjecture problem for $q$-starlike functions has been solved in the following form.

\medskip
\noindent
{\bf Theorem~A.~}\cite[Theorem~1.18]{IMS90} 
{\em If $f(z)=z+\sum_{n=2}^\infty a_nz^n$ belongs to the class $\es^*_q$, then $|a_n|\le c_n$
with equality holds for all $n$ if and only if $f(z)$ is a rotation of 
$$k_q(z):=z\exp\left[\sum_{n=1}^\infty\frac{-2\ln q}{1-q^n}z^n\right]=z+\sum_{n=2}^\infty c_n z^n,\quad z\in \D.
$$
}

\medskip
Note that the function $k_q(z)$ plays a role of the Koebe function $k(z)$. 
By differentiating once the above expression for $k_q(z)$ and comparing the
coefficients of $z^{n-1}$ in both sides, we get the recurrence relation in $c_n$:
$$c_2=\frac{-2\ln q}{1-q}~\mbox{ and }~
(n-1)c_n=\frac{-2\ln q}{1-q^{n-1}}(n-1)+\sum_{k=2}^{n-1}\frac{-2\ln q}{1-q^{k-1}}c_{n+1-k}(k-1),
\quad n\ge 3.
$$
It can be easily verified that Theorem~A turns into the famous conjecture 
of Bieberbach (known as Bieberbach-de Branges Theorem) for the class $\es^*$,
if $q\to 1^{-}$.
Comparing with the Bieberbach-de Branges theorem for close-to-convex functions,
one would expect that Theorem~A also holds true for $q$-close-to-convex functions.
However, this problem remains an open problem.
Indeed, in this manuscript, we obtain an optimal coefficient bound for $q$-close-to-convex
functions leading to the Bieberbach-de Branges 
theorem for close-to-convex functions, when $q\to 1^{-}$.
Finally, for a special attention, we collect few consequences of the 
Bieberbach-de Branges theorem for the class $\K_q$ with respect to the 
nine starlike functions considered above.

\section{Properties for $f(z)=z+\displaystyle\sum_{n=2}^\infty A_nz^n$ to be in $\K_q$}

In this section, we mainly concentrate on problems in situations where the co-efficients $A_n$
of functions $f(z)=z+\sum_{n=2}^\infty A_nz^n$ in $\K_q$ are real, non-negative and connected with certain monotone 
properties. Similar investigations for the class of close-to-convex functions are
studied in \cite{Ale15,Mac69} (see references there in for initial contributions of 
Fej\'er and Szeg\"o in this direction).

We obtain several sufficient conditions for the representation
$f(z)=z+\sum_{n=2}^\infty A_nz^n$ to be in $\K_q$. Rewriting this representation, we get
\begin{equation}\label{eq2} 
f(z)=\sum_{n=0}^\infty A_n z^n \quad (A_0=0,~A_1=1).
\end{equation}
If $f(z)$ is of the form (\ref{eq2}), then a simple computation yields
\begin{equation}\label{eq2-1}
(D_q f)(z)=1+\sum_{n=2}^\infty \frac{A_n (1-q^n)}{1-q}\, z^{n-1}
\end{equation}
for all $z\in \D$. With this, we now collect a number of sufficient conditions for functions
to be in $\K_q$.

\begin{lemma}\label{thm1}\cite[Lemma~1.1(1)]{RS12}
Let $f(z)$ be of the form $(\ref{eq2})$
and
$\sum_{n=1}^\infty |B_{n+1}-B_n|\leq 1,
$ with $B_n=A_n(1-q^n)/(1-q)$. Then $f(z)\in \K_q$ with $g(z)=z/(1-z)$.
\end{lemma}

As a consequence of Lemma~\ref{thm1}, we have

\begin{theorem}\label{thm1.1}
Let $\{A_n\}$ be a sequence of real numbers such that
$ B_n=A_n(1-q^n)/(1-q) 
$ for all $n\geq 1.$
Suppose that
$$1\geq B_2\geq B_3 \geq \cdots \geq B_n \geq \cdots \geq 0
~~\mbox{ or }~~
1\leq B_2\leq B_3 \leq \cdots \leq B_n \leq \cdots \leq 2.
$$ 
Then $f(z)=z+\sum_{n=2}^\infty A_nz^n \in \K_q$ with $g(z)=z/(1-z)$.
\end{theorem}
\begin{proof}
We know that 
$$\sum_{n=1}^\infty|B_{n+1}-B_n|=\lim_{k\rightarrow \infty}\sum_{n=1}^k|B_{n+1}-B_n|.
$$
If $1\geq B_2\geq B_3 \geq \cdots \geq B_n \geq \cdots \geq 0 , $
we see that
$$ \lim_{k\rightarrow \infty}\sum_{n=1}^k|B_{n+1}-B_n|=\lim_{k\rightarrow \infty}(B_1-B_{k+1})\leq B_1=1. 
$$
Similarly, if $1\leq B_2\leq B_3 \leq \cdots \leq B_n \leq \cdots \leq 2$ then we get
$\sum_{n=1}^\infty |B_{n+1}-B_n|\leq 1$. Thus, by Lemma~\ref{thm1}, we prove the assertion
of our theorem.
\end{proof}

\noindent
{\bf Remark.} If we choose the limit $q\to 1^{-}$ in Theorem~\ref{thm1.1}, one can obtain 
the results of Alexander~\cite{Ale15} and MacGregor~\cite{Mac69}.

\noindent
{\bf Example.} {\em The quantum dilogarithm} function is defined by 
$$Li_2(z;q)=\sum_{n=1}^\infty \frac{z^n}{n(1-q^n)},\quad |z|<1,~0<q<1.
$$
Note that this function is studied by Kirillov~\cite{Kir95} (see also \cite[p.28]{Zag07})
and is a $q$-deformation of the ordinary {\em dilogarithm function}~\cite{Kir95}
defined by $Li_2(z)=\sum_{n=1}^\infty (z^n/n^2)$, $|z|<1$, in the sense
that 
$$\lim_{\epsilon\to 0}\epsilon Li_2(z;e^{-\epsilon})=Li_2(z).
$$
By Theorem~\ref{thm1.1}, one can ascertain that the function
$(1-q)Li_2(z;q)\in \K_q$.

%


\begin{theorem}\label{thm2-1}
Let $f$ be defined by $(\ref{eq2})$ and suppose that 
$$\sum_{n=1}^\infty |B_n-B_{n-1}|\leq 1, \quad B_n=\frac{A_{n+1}(1-q^{n+1})}{1-q}
-\frac{A_n(1-q^n)}{1-q}.
$$
Then $f\in \K_q$ with $g(z)=z/(1-z)^2$.
\end{theorem}

\begin{proof}
Starting with $|B_n|$, we see that
$$|B_n|=\Big|\sum_{k=1}^{n}(B_k-B_{k-1})+1\Big|
\leq \sum_{k=1}^\infty |B_k-B_{k-1}|+1
\leq 2.
$$
Hence, for all $n\ge 2$, we have
$$\left|\frac{A_n(1-q^n)}{1-q}-\frac{A_{n-1}(1-q^{n-1})}{1-q}\right|\leq 2 .
$$
Now, by using the repeated triangle inequality, we see that
\begin{align*}
\left|\frac{A_n(1-q^n)}{1-q}\right| & = \left|\frac{A_n(1-q^n)}{1-q}-\frac{A_{n-1}(1-q^{n-1})}{1-q}
+\frac{A_{n-1}(1-q^{n-1})}{1-q}-\frac{A_{n-2}(1-q^{n-2})}{1-q}\right.\\
& \hspace*{7cm}\left. +\cdots + \frac{A_{2}(1-q^{2})}{1-q}-1+1\right|\\
& \le 2(n-1)+1 =2n-1
\end{align*}
and so
$\displaystyle |A_n|\leq 
(2n-1)/(1+q+\cdots +q^{n-1}) 
$.
By applying the root test, one can see that the radius of convergence of 
$\sum_{n=0}^\infty A_nz^n$ is not less than unity.
Therefore, $f\in \A$.

Since $f$ is of the form (\ref{eq2}), we compute by using (\ref{eq2-1}) that
\begin{align*}(1-z)^2 (D_q f)(z) 
& = 1+\frac{A_2(1-q^2)}{1-q} z -2z \\
&\quad +\sum_{n=3}^\infty \left[\frac{A_n(1-q^n)}{1-q}-\frac{2A_{n-1}(1-q^{n-1})}{1-q}
+\frac{A_{n-2}(1-q^{n-2})}{1-q}\right]z^{n-1}.
\end{align*}

%
 
By the definition of $B_n$ as given in the hypothesis, we have
$$(1-z)^2 (D_q f)(z)=1+(B_1-1)z+\sum_{n=3}^\infty (B_{n-1}-B_{n-2})z^{n-1}.
$$
Hence,
$$ \frac{1}{1-q}-\Big|(1-z)^2(D_q f)(z)-\frac{1}{1-q}\Big|
\geq 1- |B_1-1| - \sum_{n=3}^\infty |B_{n-1}-B_{n-2}|\geq 0,
$$
if $\sum_{n=2}^\infty |B_{n-1}-B_{n-2}| \leq 1.$
This proves the assertion of our theorem.
\end{proof}

By Theorem~\ref{thm2-1}, we immediately have the following result which generalizes couple 
of results of MacGregor (see \cite[Theorems~3 and 5]{Mac69}).

\begin{theorem}\label{thm2-1.1}
Let $\{A_n\}$ be a sequence of real numbers such that
$$A_0=0,~A_1=1 ~\mbox{ and }~ B_n=\frac{A_{n+1}(1-q^{n+1})}{1-q}
-\frac{A_n(1-q^n)}{1-q}.
$$
Suppose that
$$1\geq B_1\geq B_2 \geq \cdots \geq B_n \geq \cdots \geq 0
~~\mbox{ or }~~
1\leq B_1\leq B_2 \leq \cdots \leq B_n \leq \cdots \leq 2.
$$
Then $f(z)=z+\sum_{n=2}^\infty A_nz^n \in \K_q$ with $g(z)=z/(1-z)^2$.
\end{theorem}

\begin{theorem}\label{thm2-2}
Let $f$ be defined by $f(z)=z+\sum_{n=2}^\infty A_{2n-1} z^{2n-1}$ and suppose that 
$$\sum_{n=1}^\infty |B_{2n-1}-B_{2n+1}|\leq 1,\quad B_n=\frac{A_n(1-q^n)}{1-q}.
$$
Then $f\in \K_q$ with $g(z)=z/(1-z^2)$.
\end{theorem}

\begin{proof}
First of all we shall prove that $f(z)=z+\sum_{n=2}^\infty A_{2n-1} z^{2n-1}\in \A$.
For this, we estimate
$$|B_{2n+1}|  =\Big|\sum_{k=1}^{n}(B_{2k-1}-B_{2k+1})-1\Big|
 \leq 2
 $$
so that
$ \displaystyle |A_n|\leq 2/(1+q+\cdots q^{n-1}).
$
By applying the root test, one can see that the radius of convergence of the series
expansion of $f(z)$ 
is not less than unity. Therefore, $f\in \A$.

Since $f(z)=z+\sum_{n=2}^\infty A_{2n-1} z^{2n-1}$, by (\ref{q-operator}) we get
\begin{align*}
(1-z^2)(D_q f)(z)
& =1 - \sum_{n=1}^\infty \left[\frac{A_{2n-1}({1-q^{2n-1}})}{1-q}-
\frac{A_{2n+1}(1-q^{2n+1})}{1-q}\right] z^{2n} .
\end{align*}
Note that $B_n=A_n(1-q^n)/(1-q)$. So, we have
\begin{align*}
  \frac{1}{1-q}-\left|(1-z^2)(D_q f)(z)-\frac{1}{1-q}\right|
& \geq 1-\sum_{n=1}^\infty |B_{2n-1}-B_{2n+1}| \geq 0,
\end{align*}
whenever $\sum_{n=1}^\infty |B_{2n-1}-B_{2n+1}|\leq 1$. This proves the conclusion
of our theorem.
\end{proof}
By Theorem~\ref{thm2-2}, we immediately have the following result which generalizes a 
result of MacGregor (see \cite[Theorem~2]{Mac69}).

\begin{theorem}\label{lem2-2.1}
Let $\{A_n\}$ be a sequence of real numbers such that
$ B_n=A_n(1-q^n)/(1-q)
$ for all $n\geq 1.$
Suppose that
$$1\geq B_3\geq B_5 \geq \cdots \geq B_{2n-1} \geq \cdots \geq 0
~~\mbox{ or }~~
1\leq B_3\leq B_5 \leq \cdots \leq B_{2n-1} \leq \cdots \leq 2.
$$ 
Then $f(z)=z+\sum_{n=2}^\infty A_{2n-1}z^{2n-1} \in \K_q$ with $g(z)=z/(1-z^2)$.
\end{theorem}

\begin{lemma}\label{thm2-3}\cite[Lemma~1.1(4)]{RS12}
Let $f$ be defined by $(\ref{eq2})$ and suppose that 
$$\sum_{n=2}^\infty |B_{n}-B_{n-2}|\leq 1, \quad B_n=\frac{A_n(1-q^n)}{(1-q)}.
$$
Then $f\in \K_q$ with $g(z)=z/(1-z^2)$.
\end{lemma}

%
Lemma~\ref{thm2-3} leads the following sufficient conditions for functions to be in $\K_q$.
\begin{theorem}\label{thm2-3.1}
Let $\{A_n\}$ be a sequence of real numbers such that 
$$A_1=1 ~\mbox{ and }~ 
B_n=\frac{A_n(1-q^n)}{1-q}
$$ 
for all $n\geq 1.$ Suppose that
$$ 1\geq B_1+B_2\geq \cdots \geq B_{n-1}+B_n\geq \cdots \geq 0
~\mbox{ or }~
1\leq B_1+B_2\leq \cdots \leq B_{n-1}+B_n\leq \cdots \leq 2.
$$
Then $ f(z)=z+\sum_{n=2}^\infty A_nz^n \in \K_q$ with $g(z)=z/(1-z^2)$.
\end{theorem}
\begin{proof}
We know that 
$$ \sum_{n=2}^\infty |B_{n}-B_{n-2}|=\lim_{k\rightarrow \infty}\sum_{n=2}^k |B_{n}-B_{n-2}|.
$$
If $1\geq B_1+B_2\geq \cdots \geq B_{n-1}+B_n\geq \cdots \geq 0,$
we see that 
$$\lim_{k\rightarrow \infty}\sum_{n=2}^k |B_{n}-B_{n-2}|= \lim_{k\rightarrow \infty}(1-B_{k-1}-B_k)\leq 1+0=1.
$$
Similarly, if
$ 1\leq B_1+B_2\leq \cdots \leq B_{n-1}+B_n\leq \cdots \leq 2
$
then
$\sum_{n=2}^\infty |B_{n}-B_{n-2}|\leq 1$. Thus, by Theorem~\ref{thm2-3}, we complete the proof.
\end{proof}
As a consequence of Theorem~\ref{thm2-3.1}, one can obtain
the following new criteria for functions to be in the close-to-convex
family.

\begin{theorem}
Let $\{a_n\}$ be a sequence of real numbers such that
$a_1=1 ~\mbox{ and }~ b_n=na_n
$
for all $n\geq 1.$
Suppose that
$$ 1\geq b_1+b_2\geq \cdots \geq b_{n-1}+b_n\geq \cdots \geq 0
~~\mbox{ or }~~
1\leq b_1+b_2\leq \cdots \leq b_{n-1}+b_n\leq \cdots \leq 2.
$$
Then $f(z)=z+\sum_{n=2}^\infty a_nz^n$ is close-to-convex
with $g(z)=z/(1-z^2)$. 
\end{theorem}

\begin{lemma}\label{thm3}\cite[Lemma~1.1(2)]{RS12}
Let $f$ be defined by $(\ref{eq2})$ and suppose that 
$$\sum_{n=1}^\infty |B_{n-1}-B_n+B_{n+1}|\leq 1, \quad
B_n=\frac{A_n(1-q^n)}{1-q}.
$$
Then $f\in \K_q$ with $g(z)=z/(1-z+z^2)$.
\end{lemma}

Lemma~\ref{thm3} yields the following sufficient condition. 
\begin{theorem}\label{thm3.1}
Let $\{A_n\}$ be a sequence of real numbers such that
$$A_1=1 ~\mbox{ and }~ 
B_n=\frac{A_n(1-q^n)}{1-q}
$$
for all $n\geq 1.$
Suppose that
$$ 0\geq B_2-B_1\geq B_3\geq B_2+B_4\geq B_2+B_3+B_5\geq \cdots \geq B_2+B_3+B_4+\cdots +B_{n-1}+B_{n+1}\geq -1
$$
 or
$$ 0\leq B_2-B_1\leq B_3\leq B_2+B_4\leq B_2+B_3+B_5\leq \cdots \leq B_2+B_3+B_4+\cdots +B_{n-1}+B_{n+1}\leq 1
$$ 
holds.
Then $ f(z)=z+\sum_{n=2}^\infty A_nz^n \in \K_q$ with $g(z)=z/(1-z+z^2)$.
\end{theorem}
\begin{proof}
We know that 
$$ \sum_{n=1}^\infty |B_{n-1}-B_n+B_{n+1}|=\lim_{k\rightarrow \infty}\sum_{n=1}^k |B_{n-1}-B_n+B_{n+1}|.
$$
If 
$$ 0\geq B_2-B_1\geq B_3\geq B_2+B_4\geq B_2+B_3+B_5\geq \cdots \geq B_2+B_3+B_4+\cdots +B_{n-1}+B_{n+1}\geq -1,
$$
we see that 
$$\lim_{k\rightarrow \infty}\sum_{n=1}^k |B_{n-1}-B_n+B_{n+1}|
= \lim_{k\rightarrow \infty}-(B_2+B_3+B_4+\cdots +B_{k-1}+B_{k+1})\leq 1.
$$
Similarly, if
$$ 0\leq B_2-B_1\leq B_3\leq B_2+B_4\leq B_2+B_3+B_5\leq \cdots \leq B_2+B_3+B_4+\cdots +B_{n-1}+B_{n+1}\leq 1
$$ 
then one can obtain
$\sum_{n=1}^\infty |B_{n-1}-B_n+B_{n+1}|\leq 1$. Thus, by Theorem~\ref{thm3},
we complete the proof.
\end{proof}

As a result of Theorem~\ref{thm3.1}, one can obtain
the following new criteria for functions to be in the close-to-convex
family.

\begin{theorem}
Let $\{a_n\}$ be a sequence of real numbers such that
$a_1=1 ~\mbox{ and }~ b_n=na_n,
$
for all $n\geq 1.$
Suppose that
$$ 0\geq b_2-b_1\geq b_3\geq b_2+b_4\geq b_2+b_3+b_5\geq \cdots \geq b_2+b_3+b_4+\cdots +b_{n-1}+b_{n+1}\geq -1
$$
 or
$$ 0\leq b_2-b_1\leq b_3\leq b_2+b_4\leq b_2+b_3+b_5\leq \cdots \leq b_2+b_3+b_4+\cdots +b_{n-1}+b_{n+1}\leq 1.
$$ 
 Then $\displaystyle f(z)=z+\sum_{n=2}^\infty a_nz^n$ is in the close-to-convex
 family with respect to $g(z)=z/(1-z+z^2)$.
\end{theorem}

\section{The Bieberbach-de Branges Theorem for $\K_q$}

A necessary and sufficient condition for functions to be in
$\es_q^*$ was obtained in \cite[Theorem~1.5]{IMS90} in
the following form: 
{\em a function $f\in \es_q^*$ if and only if $|f(qz)/f(z)|\leq 1$ for all $z\in \D$.}

A similar characterization for functions in $\K_q$ is

\begin{lemma}\label{new-lem1}
A function $f\in \K_q$ if and only if there exists $g\in \es^*$ such that
$$ \frac{|g(z)+f(qz)-f(z)|}{|g(z)|}\le 1 \quad \mbox{ for all $z\in \D$.}
$$
\end{lemma}
\begin{proof}
The proof follows immediately after making the substitution for the expression
of the $q$-difference operator $(D_qf)(z)$ in Definition~\ref{def2}.
\end{proof}

In this section, Lemma~\ref{new-lem1} will act as one of the crucial results
to estimate coefficient bounds for series representation of functions in the class $\K_q$,
i.e. in other words, we analyze the Bieberbach-de Branges theorem 
for the class of $q$-close-to-convex functions. The Bieberbach conjecture
for close-to-convex functions is proved by Reade \cite{Rea55} (see also \cite{Good83}
for more details). It states that {\em if $f\in \K$, then $|a_n|\le n$ for all $n\geq 2$}.
%
%
%
%

We now proceed to state and prove the Bieberbach-de Branges Theorem for functions in
the $q$-close-to-convex family.
\begin{theorem}[Bieberbach-de Branges Theorem for $\K_q$]\label{new-thm2}
If $f\in \K_q$, then
$$|a_n|\leq \frac{1-q}{1-q^n}\left[n+\frac{n(n-1)}{2}(1+q)\right] \quad \mbox{ for all $n\ge 2$.}
$$
\end{theorem}
\begin{proof}
Since $f\in\K_q$, by Lemma~\ref{new-lem1} there exists $w:\,\D \to \overline{\D}$ such that
\begin{equation}\label{eqn1-sec3}
g(z)+f(qz)-f(z)=w(z)g(z).
\end{equation}
Clearly $w(0)=q$. By assuming $a_1=1=b_1$, we then have
$$\sum_{n=1}^\infty (b_n+a_nq^n-a_n)z^n = \sum_{n=1}^\infty q b_n z^n 
+ \sum_{n=2}^\infty \left(\sum_{k=1}^{n-1}w_{n-k}b_k\right)z^n .
$$
Equating the coefficients of $z^n$, for $n\geq 2$, we obtain
$$a_n(q^n-1)=b_n(q-1)+\sum_{k=1}^{n-1}w_{n-k}b_k .
$$
From the classical result \cite{DCP11}, one can verify that 
$|w_n|\leq 1-|w_0|^2=1-q^2$ for all $n\ge 1$. 
Since $g(z)=z+\sum_{n=2}^\infty b_nz^n\in S^*$, 
we get
\begin{align*}
|a_n| & \leq \frac{1-q}{1-q^n}\left[n+(1+q)\sum_{k=1}^{n-1}k\right] \quad \mbox{ for all $n\geq 2$}.
\end{align*}
This proves the conclusion of our theorem.
\end{proof}

\noindent
{\bf Remark.}
When $q\to 1^{-}$, certainly Theorem~\ref{new-thm2} yields the 
Bieberbach conjecture problem for close-to-convex functions.

It is easy to see, by the usual ratio test, that the series 
\begin{equation}\label{series1}
z+\sum_{n=2}^\infty \frac{1-q}{1-q^n}\left[n+\frac{n(n-1)}{2}(1+q)\right]z^n
\end{equation}
converges for $|z|<1$. Indeed, we can ascertain by using the convergence factor
for the series $\sum_{n=1}^\infty z^n/(1-q^n)$ (see \cite[3.2.2.1]{Sla66})
that the series given by (\ref{series1}) converges to the function
$$\frac{1+q}{2}z^2\frac{d^2\Psi(q;z)}{dz^2}+z\frac{d\Psi(q;z)}{dz}\,,
$$
where $\Psi(q;z):=z\Phi[q,q;q^2;q,z]$ represents its Heine hypergeometric function.
Note that the $q$-hypergeometric series was developed by Heine~\cite{Hei46}
as a generalization of the well-known Gauss hypergeometric series:
$$\Phi[a,b;c;q,z]=\sum_{n=0}^\infty \frac{(a;q)_n(b;q)_n}{(c;q)_n(q;q)_n}z^n,\quad 
|q|<1, ~1\neq cq^n, ~|z|<1,
$$
where the $q$-shifted factorial $(a;q)_n$ is defined by 
$$(a;q)_n=(1-a)(1-aq) \cdots (1-aq^{n-1}) ~~\mbox{ and }~~ (a;q)_0=1.
$$
This is also known as the basic hypergeometric series and its convergence function
is known as the basic hypergeometric function. 
We refer one of the standard books \cite{Sla66} for the notation of the basic 
hypergeometric function. For history of $q$-series related calculus 
and their applications, we suggest readers to refer \cite{Ern02}. 

Due to Frideman's result, we now study the special cases of Theorem~\ref{new-thm2}
with respect to the nine functions having integer coefficients.
However, in this situation, it is enough to consider the identity function and 
four other functions which contain factors $1-z$ instead of $1\pm z$ in the denominator.
In particular, Theorem~\ref{new-thm2} reduces to the following
corollaries. Note that we provide proofs of last two consequences
as they involve variations in the exponents, whereas the first three
consequences follow directly after making precise substitution for 
the starlike functions $g(z)$.

\begin{corollary}
If $f\in \K_q$ with the Koebe function $g(z)=z/(1-z)^2$, then for all $n\geq 2$ we have 
$$|a_n| \leq \frac{1-q}{1-q^n}\left[n+(1+q)\,\frac{n(n-1)}{2}\right].
$$
\end{corollary}

If $f\in \K$ with $g(z)=z$, then for all $n\ge 2$ it is well-known that $|a_n|\le 2/n$. 
As a generalization, we have the following:
\begin{corollary}
If $f\in \K_q$ with $g(z)=z$, then for all $n\geq 2$ we have $|a_n|\le (1-q^2)/(1-q^n)$.
\end{corollary}
Here we note that the series $z+\sum_{n=2}^\infty (1-q^2)/(1-q^n)z^n$ converges to
the Heine hypergeometric function $(z+qz)\Phi[q,q;q^2;q,z]-qz
=z+z^2\Phi[q^2,q^2;q^3;q^2,z]$ and it follows from
\cite[3.2.2, pp.~91]{Sla66}.

If $f\in \K$ with $g(z)=z/(1-z)$, then for all $n\ge 2$ it is known that $|a_n|\le (2n-1)/n$.
We find the following analogous result:
\begin{corollary}
If $f\in \K_q$ with $g(z)=z/(1-z)$, then for all $n\geq 2$ we have 
$$|a_n|\le \frac{1-q}{1-q^n} \,[n+q(n-1)].
$$
\end{corollary}
One can similarly verify that the series
$z+\sum_{n=2}^\infty \frac{1-q}{1-q^n}[n+q(n-1)]$
converges to the function $z(1+q)\frac{d}{dz}\Psi(q;z)-q\Psi(q;z)$,
where $\Psi(q;z):=z\Phi[q,q;q^2;q,z]$ represents its Heine hypergeometric function.

If $f\in \K$ with $g(z)=z/(1-z^2)$, then for all $m\ge 1$ it is known that 
$$|a_n|\leq \left \{
\begin{array}{ll}
1, & \mbox{ if } n=2m-1;\\
1, & \mbox{ if } n=2m.
     \end{array}\right..
$$
As a generalization, we now state the following corollary along with an outline of its proof:
\begin{corollary}
If $f\in \K_q$ with $g(z)=z/(1-z^2)$, then for all $m\geq 1$ we have 
$$|a_n|\leq \left \{
\begin{array}{ll}
\displaystyle{\frac{1-q}{1-q^n}\left(\frac{n}{2}(1+q)+\frac{1}{2}(1-q)\right)}, 
& \mbox{ if } n=2m-1;\\[4mm]
\displaystyle{\Big(\frac{1-q^2}{1-q^n}\Big)\frac{n}{2}}, & \mbox{ if } n=2m. \end{array}\right..
$$
\end{corollary}

\begin{proof}
Since $\displaystyle{g(z)=\frac{z}{1-z^2}}=\sum_{n=1}^{\infty}z^{2n-1}$, by (\ref{eqn1-sec3})
we get
$$ \sum_{n=1}^{\infty}(q^n-1)a_n z^n=(q-1)\sum_{n=1}^{\infty}z^{2n-1}+\left(\sum_{n=1}^{\infty}z^{2n-1}\right)
\left(\sum_{n=1}^{\infty}w_n z^n\right).
$$
This is equivalent to
\begin{equation}\label{new-eq3}
 \sum_{n=1}^{\infty}(q^n-1)a_n z^n=(q-1)\sum_{n=1}^{\infty}z^{2n-1}+\sum_{n=2}^{\infty}\left(\sum_{k=1}^{n-1}w_{2k} \right)z^{2n-1}
+\sum_{n=1}^{\infty}\left(\sum_{k=1}^{n}w_{2k-1} \right)z^{2n}.
\end{equation}
In order to prove the required optimal bound for $|a_n|$, in this situation,
it is appropriate to compare the coefficients of $z^{2n-1}$ and $z^{2n}$ separately.

In (\ref{new-eq3}), first we compare the coefficients of $z^{2n-1},$ for $n\geq 2$, we get 
$$(q^{2n-1}-1)a_{2n-1}=(q-1)+\sum_{k=1}^{n-1}w_{2k}.
$$
Since $|w_k|\leq (1-q^2)$ for all $k\geq 1$  and $q\in (0,1),$ we have
$$ |a_{2n-1}|\leq \frac{1-q}{(1-q^{2n-1})}\left(-q+(1+q)n\right).
$$
Secondly, by comparing the coefficients of $z^{2n},$ for $n\geq 1,$ we obtain
$$(q^{2n}-1)a_{2n}=\sum_{k=1}^{n}w_{2k-1},
$$
and similarly we get the bound
$$ |a_{2n}|\leq \frac{1-q}{(1-q^{2n})}(1+q)n.
$$
Thus, we prove the required optimal bound for $|a_n|$.
\end{proof}

If $f\in \K$ with $g(z)=z/(1-z+z^2)$, then for all $n\ge 2$ it is known that 
$$|a_n|\leq \left \{
\begin{array}{lll}
\displaystyle{\frac{4n+1}{3n}}, & \mbox{ if } n=3m-1;\\[3mm]
\displaystyle{\frac{4}{3}}, & \mbox{ if } n=3m; \\[4mm]
\displaystyle{\frac{4n-1}{3n}}, & \mbox{ if } n=3m+1.
\end{array}\right..
$$
As a generalization, we have the following:
\begin{corollary}
If $f\in \K_q$ with $g(z)=z/(1-z+z^2)$, then for all $m\geq 1$ we have 
$$|a_n| \leq \left \{
\begin{array}{ll}
\displaystyle{\frac{1-q}{1-q^n}\left(\frac{1}{3}(2-q)+\frac{2n}{3}(1+q)\right)}, 
& ~\mbox{if}~\ n=3m-1;\\[4mm]
\displaystyle{\frac{1-q^2}{1-q^n}\frac{2n}{3}}, & ~\mbox{if}~\ n=3m; \\[4mm]
\displaystyle{\frac{1-q}{1-q^n}\left(\frac{2n}{3}(1+q)+\frac{1}{3}(1-2q)\right)}, 
& ~\mbox{if}~\ n=3m+1.
     \end{array}
\right..
$$
\end{corollary}

\begin{proof}
By rewriting the function $\displaystyle g(z)=z/(1-z+z^2),$ we obtain 
$$g(z)=\frac{z(1+z)}{1+z^3}=\sum_{n=1}^{\infty}(-1)^{n-1}z^{3n-2}
+\sum_{n=1}^{\infty}(-1)^{n-1}z^{3n-1}.
$$
Then simplifying the relation (\ref{eqn1-sec3}), we get
\begin{eqnarray}\label{new-eq4}
 \nonumber&&\hspace*{-1.5cm}\sum_{n=1}^{\infty}(q^n-1)a_n z^n\\[-3mm]
\nonumber &=&(q-1)\left(\sum_{n=1}^{\infty}(-1)^{n-1}z^{3n-2}+
 \sum_{n=1}^{\infty}(-1)^{n-1}z^{3n-1}\right)\\
\nonumber & \quad +&\sum_{n=1}^{\infty}\left(\sum_{k=1}^{n}(-1)^{n-k}w_{3k-2} \right)z^{3n-1}
 +\sum_{n=1}^{\infty}\left(\sum_{k=1}^{n}(-1)^{n-k}w_{3k-1} \right)z^{3n}\\
\nonumber & \quad +&\sum_{n=1}^{\infty}\left(\sum_{k=1}^{n}(-1)^{n-k}w_{3k} \right)z^{3n+1}
+\sum_{n=2}^{\infty}\left(\sum_{k=1}^{n-1}(-1)^{n-k}w_{3k} \right)z^{3n-1}\\
  & \quad +&\sum_{n=1}^{\infty}\left(\sum_{k=1}^{n}(-1)^{n-k}w_{3k-2} \right)z^{3n}
+\sum_{n=1}^{\infty}\left(\sum_{k=1}^{n}(-1)^{n-k-1}w_{3k-1} \right)z^{3n+1}.
\end{eqnarray}

First equating the coefficients of $z^{3n-1},$ for $n\geq 2,$ in (\ref{new-eq4}), we get
$$(q^{3n-1}-1)a_{3n-1}=(-1)^{n-k}(q-1)+\sum_{k=1}^{n}(-1)^{n-k}w_{3k-2}
 +\sum_{k=1}^{n}(-1)^{n-k}w_{3k}.
$$
Since $|w_k|\leq (1-q^2)$ for all $k\geq 1$ and $q\in (0,1),$ we have
$$|a_{3n-1}|\leq \frac{1-q}{(1-q^{3n-1})}\left(-q+2(1+q)n)\right).
$$
Next, for all $n\geq 1$, we compare the coefficients of $z^{3n}$ and $z^{3n+1}$  
in (\ref{new-eq4}), we respectively obtain the coefficient bounds
$$|a_{3n}|\leq \frac{2(1-q)}{(1-q^{3n})}(1+q)n
~\mbox{ and }~
|a_{3n+1}|\leq \frac{(1-q)}{(1-q^{3n+1})}(1+2(1+q)n).
$$
Thus, the assertion of our corollary follows.
\end{proof}

\noindent
{\bf Remark.}
By making use of Lemma~\cite[Theorem~1.5]{IMS90}, one can also obtain
the Bieberbach-de Branges theorem for $\es^*_q$ as follows. This also
yields the Bieberbach-de Branges theorem for $\es^*$, in particular. 
However, it defers from Theorem~A. 

\section{Appendix}
In this section, we verify that a similar technique used in the previous section yields
a form of the Bieberbach-de Branges theorem for $\es^*_q$. 
This leads to the coefficient problem of Bieberbach-de Branges 
(different from Theorem~A\,!) for the class $\es^*$, when $q\to 1^{-}$, as well.

\begin{theorem}[The Bieberbach-de Branges Theorem for $\es_q^*$]\label{new-thm1}
If $f\in \es_q^*$, then for all $n\geq 2$ we have
\begin{equation}\label{new-eq1}
|a_n|\leq \displaystyle{ \left(\frac{1-q^2}{q-q^n}\right)
{\prod_{k=2}^{n-1}}}\left(1+\frac{1-q^2}{q-q^{k}}\right).
\end{equation}
\end{theorem}
\begin{proof}
We know that $f\in \es_q^*$ if and only if 
$$|f(qz)/f(z)|\le 1 \quad \mbox{ for all $z\in \D$.}
$$
Then there exists $w:\,\D \to \overline{\D}$ such that
$$\frac{f(qz)}{f(z)}=w(z), \quad \mbox{i.e. $f(qz)=w(z)f(z)$ for all $z\in \D$.}
$$
Clearly, $w(0)=q$. In terms of series expansion, we get (with $a_1=1$ and $w_0=q$)
$$\sum_{n=1}^\infty a_n q^n z^n = \left(\sum_{n=0}^\infty w_nz^n\right)
\left(\sum_{n=1}^\infty a_nz^n\right)=:\sum_{n=1}^\infty c_nz^n,
$$
where $c_n:=\sum_{k=1}^n w_{n-k}a_k=qa_n+\sum_{k=1}^{n-1}w_{n-k}a_k$.
Comparing the coefficients of $z^n$ ($n\ge 2$), we get
$$a_n(q^n - q) = \sum_{k=1}^{n-1} w_{n-k}a_k, \quad\mbox{ for $n\geq 2$} .
$$
Since $|w_n|\le 1-|w_0|^2=1-q^2$ for all $n\ge 1$, we see that
$$|a_n|  \leq \frac{1-q^2}{q-q^n} \,\sum_{k=1}^{n-1} |a_k| \quad \mbox{ for each $n\geq 2$.}
$$
Thus for $n=2$, one has $|a_2|\le (1-q^2)/(q-q^2)$, and for $n\ge 3$, we apply
a similar technique to estimate $|a_{n-1}|$ and get
$$|a_n|\leq \frac{1-q^2}{q-q^n}\left(1+\frac{1-q^2}{q-q^{n-1}}\right)\,\sum_{k=1}^{n-2}|a_k| .
$$
Iteratively, we conclude that
$$|a_n|\leq \frac{1-q^2}{q-q^n}\left(1+\frac{1-q^2}{q-q^{n-1}}\right)
\left(1+\frac{1-q^2}{q-q^{n-2}}\right)\cdots \left(1+\frac{1-q^2}{q-q^{2}}\right)
$$
for all $n\geq 3.$ This completes the proof.
\end{proof}

\noindent
{\bf Remark.}
One can easily verify that the right hand side of (\ref{new-eq1}) approaches $n$
as $q\to 1^{-}$, which will lead to the Bieberbach-de Branges theorem 
for starlike functions \cite[Theorem~2.14]{Dur83}. 


We also find that the ratio test easily provides the convergence of the series 
$z+\sum_{n=2}^\infty A_nz^n$ in the sub-disk $|z|<q/(q+1-q^2)$, where 
$$A_n=\left(\frac{1-q^2}{q-q^n}\right)
{\prod_{k=1}^{n-2}}\left(1+\frac{1-q^2}{q-q^{k+1}}\right).
$$

\noindent
{\bf Acknowledgement.} The work of the second author is supported by
the National Board for Higher Mathematics, Department of Atomic Energy, India.
Both the authors thank Professor S. Ponnusamy for bringing 
the article \cite{IMS90} to their attention and useful discussion on this topic.
The authors also thank the referee for his/her careful reading of the manuscript
and valuable comments. The authors also acknowledge the help of Dr. Amarjeet Nayak 
in rectifying some language-related issues with the paper.

\end{document}